\documentclass[letterpaper,onecolumn]{ieeeconf}

\IEEEoverridecommandlockouts                              % This command is only needed if 
\usepackage{graphics} % for pdf, bitmapped graphics files
\usepackage{epsfig} % for postscript graphics files
\usepackage{mathptmx} % assumes new font selection scheme installed
\usepackage{times} % assumes new font selection scheme installed
\usepackage{amsmath,amssymb,amsfonts}
\usepackage{algorithmic}
\usepackage{graphicx}% Required for inserting images
\usepackage{textcomp}
\usepackage{xcolor}
\usepackage{booktabs}
\usepackage{amsfonts}
\usepackage{hyperref}
\usepackage{bm}
\usepackage[english]{babel}

\usepackage{amsthm}
\usepackage{mathtools} % for  := 
\usepackage{arydshln}
\newtheorem{definition}{Definition}
\newtheorem{theorem}{Theorem}
\newtheorem{lemma}{Lemma}

\theoremstyle{remark}

\newcommand{\R}{\mathbb{R}}

\newcommand{\E}{\mathbb{E}}
\newcommand{\bmtx}{\begin{bmatrix}}
\newcommand{\emtx}{\end{bmatrix}}
\newcommand{\bsmtx}{\left[ \begin{smallmatrix}} 
\newcommand{\esmtx}{\end{smallmatrix} \right]}
\newcommand{\bmatarray}[1]{\left[\begin{array}{#1}}
\newcommand{\ematarray}{\end{array}\right]}

\title{\LARGE \bf
Regret of $H_\infty$ Preview Controllers
}

\author{Jietian Liu$^{1}$ and Peter Seiler$^{1}$% <-this % stops a space
\thanks{$^{1}$Jietian Liu and Peter Seiler are with the Department of Electrical Engineering and Computer Science, University of Michigan,
Emails: {\tt\small jietian@umich.edu} and {\tt\small pseiler@umich.edu}}%
}

\begin{document}

\maketitle
\thispagestyle{empty}
\pagestyle{empty}

%%%%%%%%%%%%%%%%%%%%%%%%%%%%%%%%%%%%%%%%%%%%%%%%%%%%%%%%%%%%%%%%%%%%%%%%%%%%%%%%
\begin{abstract}
This paper studies preview control in both the $H_\infty$ and regret-optimal settings. The plant is modeled as a discrete-time, linear time-invariant system subject to external disturbances. The performance baseline is the optimal non-causal controller that has full knowledge of the disturbance sequence. We first review the construction of the $H_\infty$ preview controller with $p$-steps of disturbance preview. We then show that the closed-loop $H_\infty$ performance of this preview controller converges as $p\to \infty$ to the performance of the optimal non-causal controller. Furthermore, we prove that the optimal regret of the preview controller converges to zero. These results demonstrate that increasing preview length allows controllers to asymptotically achieve non-causal performance in both the $H_\infty$ and regret frameworks. A numerical example illustrates the theoretical results.
\end{abstract}

%%%%%%%%%%%%%%%%%%%%%%%%%%%%%%%%%%%%%%%%%%%%%%%%%%%%%%%%%%%%%%%%%%%%%%%%%%%%%%%%
\section{INTRODUCTION}
This paper considers preview control in both the $H_\infty$ and regret-optimal settings for discrete-time, linear time-invariant (LTI) systems subject to external disturbances. The classical $H_\infty$ formulations \cite{zhou96,dullerud99} evaluate performance by minimizing the worst case of closed-loop gain from disturbance to error. These approaches provide powerful worst-case guarantees but do not explicitly account for the advantage of preview information, nor do they compare performance against more informative non-causal benchmarks.

Preview control arises in applications where sensors provide advance measurements of external disturbances. For example, wind turbines equipped with LIDAR sensors can measure the incoming wind field to improve power capture and reduce structural loads \cite{ozdemir13,schlipf13,scholbrock16}, while forward-looking sensors in active vehicle suspensions provide preview of the upcoming road profile \cite{Tomizuka1976OptimalLinearPreviewControl,Hac1992OptimalLinearPreviewControl,ROH1999StochasticOptimalPreviewControl,Louam1988OptimalControl,YOSHIMURA1993AnActiveSuspensionModel,MARZBANRAD2004StochasticOptimalPreviewControl}. In these settings, future disturbances cannot be eliminated, but their effects can be anticipated and mitigated. This motivates robust preview formulations, such as $H_\infty$ preview control, which exploit finite preview information while providing worst-case performance guarantees under uncertainty.

An alternative is regret-optimal control, which measures the performance of a causal controller relative to a baseline controller with greater information or capability. Regret is defined as the performance difference between the
closed-loop cost of the causal controller and that of the baseline
controller. Typical measures for this performance difference include
additive regret and competitive ratio. Prior works have considered baselines such as the optimal non-causal controller with full disturbance preview \cite{goel22TAC,goel19PMLR,goel22CDC,goel21PMLR,sabag21ACC,sabag22CDC,goel23TAC,goel23PMLR,goel22arXivGH, Zhou23CDC, Didier22ICSL, Martin22PMLR, MARTIN2023IFAC, BROUILLON2023IFAC, MARTIN2023IFAC, Martinelli24ECC,Tsukamoto24CDC,KarapetyanEtAl22CDC}, or the best static state-feedback law \cite{agarwal19,hazan16, li2021safeadaptivelearningbasedcontrol}. Recently, several studies have extended regret-optimal control to uncertain settings under the umbrella of robust or distributionally robust regret control\cite{kargin24pmlr,kargin24pmlrIH,hajar23Allerton,taha23CDC,YAN25IFAC,Cho24LCSS,agarwal22pmlr,bitar2024distributionallyrobustregretminimization,fiechtner2025wassersteindistributionallyrobustregret}. These formulations typically neglect the preview information.

Recent work on regret-based control has also considered finite disturbance preview, for example in predictive or MPC-based frameworks \cite{Zhang2021,Lin2021}. In these approaches, preview information is incorporated through receding-horizon optimization over a finite window. In contrast, the preview formulation considered in this paper is defined with respect to an infinite-horizon cost, with preview appearing only as an information constraint.

There is a large literature on finite preview control for $H_\infty$ formulation in both finite horizon and infinite horizon \cite{Wang_Zhang_Xie_2014,Hazell_Limebeer_2008,Liu_Lan_2024,KATAYAMA_HIRONO_1987,Katoh_2004,Kojima_2005,Kojima_2015,Kojima_Ishijima_2003,Kojima_Ishijima_2004,Kojima_Ishijima_2006,TadmorMirkin05TAC1,TadmorMirkin05TAC2,Grimble91TSP,BolzernEtAl04TSP,KojimaIshijima95TAC,Tadmor97TAC,MIRKIN03Automatica,MirkinTadmor04ACC,MirkinMeinsma02IFAC,WangEtAl13CCC,ZhangEtAl12CCC,WangEtAl10ICA}. A common approach in discrete-time is to construct an augmented system with a chain
of delays to store the disturbance preview information. This technique simplifies the preview control design as the standard
$H_\infty$ solutions can be applied to this augmented model.

In this paper, we study the interplay between $H_\infty$ preview control and regret-optimal control with preview. We first revisit the standard synthesis of $p$-step $H_\infty$ preview controllers using state augmentation. We then establish two key convergence results: (i) the $H_\infty$ preview controller achieves closed-loop performance that converges to the non-causal $H_\infty$ bound as the preview horizon grows, and (ii) the optimal regret of the $p$-step preview controller converges to zero as $p\to\infty$. Thus, preview controllers asymptotically match the performance of the optimal non-causal controller. Preview controllers require additional
sensors to provide the preview measurements. The benefit is
that the preview controllers are LTI with a well-understood theory. In contrast, more recent regret-based controllers converge to the optimal
non-causal controller over time by typically relying on time-varying, online optimization. These regret-based controllers do not require
additional preview sensors but are time-varying, in general,
and hence less understood.

\section{PROBLEM FORMULATION}

Consider a discrete-time, linear time-invariant (LTI) plant $P$ with the following state-space representation:
\begin{align}
  \label{eq:Pic}
   x(t+1)= Ax(t)+B_dd(t) + B_uu(t),
\end{align}
where $x(t) \in \R^{n_x}$ is the state, $d(t) \in \R^{n_d}$ is the
disturbance, and $u(t) \in \R^{n_u}$ is the control input at time $t$, respectively.

We consider controllers with access to state measurements and finite preview of the disturbance 
$d(t)$. Specifically, the controller has the following information at time $t$:
\begin{align}
\label{eq:IHinfo}
\begin{split}    
    i_p(t) &  :=   \{x(t),d(t), \dots,d(t+p)\}, 
\end{split}
\end{align}
where $p\geq0$ denotes the number of preview steps. The case $p=0$ corresponds to no preview and is known as full information \cite{zhou96,dullerud99}, while $p \to \infty$ corresponds to complete knowledge of all future disturbances.

The goal is to design a $p$-step preview controller $K_p$ to stabilize the plant and ensure good performance as measured with a linear quadratic cost. In particular, the cost for a controller $K_p$ evaluated on a specific disturbance $d\in \ell_2$ with initial condition $x(0)=0$ is defined as follows:
\begin{align}
\label{eq:Jinf}
J(K_p,d) :=  
\sum_{t=0}^\infty x(t)^\top Qx(t)+u(t)^\top R u(t).
\end{align}
Additional technical assumptions on the cost and state matrices will be given below to ensure the optimal control problem is well-posed.

We will use the optimal non-causal controller, denoted $K_{nc}$, as a
baseline for comparison, following 
\cite{goel22TAC,goel22CDC,goel21PMLR,sabag21ACC,sabag22CDC,goel23TAC}. The
controller $K_{nc}$ depends directly on the state as well as the present and all future values of $d$. The controller $K_{nc}$ is optimal in the sense that it minimizes $J(K,d)$ for each $d\in \ell_2$. A solution for the optimal non-causal controller is given by Theorem 11.2.1 of
\cite{hassibi99}. Related  non-causal results (both finite and infinite horizon) are given in \cite{goel22TAC,goel22CDC,goel21PMLR,sabag21ACC,sabag22CDC,goel23TAC} where the non-causal controller is used as a baseline for regret-based control design. We focus on the infinite-horizon case where the non-causal controller has access to the following
information at time $t$:
\begin{align}
\label{eq:NCinfo}
    i_{nc}(t) &  :=   \{x(t),d(t), d(t+1),\ldots\}. 
\end{align} 
The optimal non-causal controller for the infinite-horizon
case can be constructed using the stabilizing solution $X$ of a discrete-time algebraic Riccati equation (DARE). Specifically, the DARE associated with the data $(A,B,Q,R)$ is defined as follows:
\begin{align}
    \label{eq:DARE}
    0 & = X  - A^\top X A  - Q 
     + A^\top X B (R+B^\top X B)^{-1} B^\top X A.
\end{align}
The next theorem provides a state-space model of $K_{nc}$.

\begin{theorem}
\label{thm:NC}
Let $(A,B_u,B_d,Q,R)$ be given and assume: 
(i) $Q\succeq 0$ and $R \succ 0$, (ii) $\left(A, B_u\right)$ stabilizable, (iii) $A$ is nonsingular, and (iv) $(A, Q)$ has no unobservable modes on the unit circle. Then:
\begin{enumerate}
%\item  $(R+B_u^\top XB_u)$ is nonsingular.
\item 
There is a unique stabilizing solution $X \succeq 0$ to DARE~\eqref{eq:DARE} defined by $(A,B_u,Q,R)$.
\item Define the gain
$K_x :=  (R+B_u^\top X B_u)^{-1}B_u^\top XA$.
Then $\tilde{A}  :=  A - B_u K_x$ is Schur and nonsingular.
\item  Define a non-causal controller $K_{nc}$ with inputs $i_{nc}(t)$ given in \eqref{eq:NCinfo} and output $u_{nc}(t)$ by the following non-causal state-space model:
\begin{align}
  \label{eq:Knc}
  \begin{split}
    v_{nc}(t) & = \tilde{A}^\top [ v_{nc}(t+1) + X B_d d(t) ], 
    \,\,\, v_{nc}(\infty)=0, \\
    u_{nc}(t) & =  -K_x x(t) - K_v v_{nc}(t+1) - K_d  d(t),
  \end{split}
\end{align}
where 
\begin{align*}
  K_v &  :=   (R+B_u^\top X B_u)^{-1} B_u^\top, \\
  K_d &  :=   (R+B_u^\top X B_u)^{-1} B_u^\top  X B_d.
\end{align*}
Then $J(K_{nc},d) \le J(K,d)$ for any stabilizing controller $K$ and
disturbance $d\in\ell_2$.
\end{enumerate}
\end{theorem}
The Riccati-based results in items (1) and (2) of Theorem~\ref{thm:NC} are classical and follow from standard discrete-time LQR theory; see, e.g., \cite{hassibi99, zhou96,ChanEtAl84TAC}. 
The specific statement of Theorem~\ref{thm:NC} follows as a special case of Lemma 1 and Theorem 1 in \cite{liu2024robust}.\footnote{The results in \cite{liu2024robust} include a cross term $x^\top S u$ in the cost function.  Theorem~\ref{thm:NC} in this paper follows immediately by setting $S=0$.} The assumption that 
$A$ is nonsingular ensures that the closed-loop matrix $\tilde{A}$ is also nonsingular.
Note that the non-causal controller \eqref{eq:Knc} can be reformulated so that the control input $u_{nc}(t)$ is expressed directly in terms of $i_{nc}(t)$ as defined in \eqref{eq:NCinfo}. Unraveling the non-causal dynamics and noting that $K_d=K_vXB_d$ yields:
\begin{align}
   u_{nc}(t) = -K_x x(t)- K_v\sum_{j=t}^\infty \left( \tilde{A}^\top \right)^{j-t}X B_d d(j)
\end{align}

Finally, we define the performance of any $p$-step preview controller $K_p$ relative to the baseline, non-causal controller $K_{nc}$ as follows:

\begin{definition}
\label{def:nomregret}
Let $\gamma > 0$ be given.  A $p$-step preview controller $K_p$
achieves $\gamma$-regret relative to the optimal
non-causal controller $K_{nc}$ if 
the closed-loop is stable and:
\begin{align}
  \label{eq:robregret}
  \begin{split}
  J(K_p,d)-J(K_{nc},d) < \gamma^2 \| d \|_2^2 
  \qquad \forall d \in \ell_2, \, d\ne 0.
  \end{split}
\end{align}
\end{definition}
Let $\gamma_{R,p}$ denote the minimal $\gamma$-regret achieved with $p$ steps of preview. We will show that $\gamma_{R,p} \to 0$ as $p\to \infty$.

\section{BACKGROUND: $H_\infty$ CONTROL WITH PREVIEW}
\label{sec:HinfPreview}

This section reviews the $p$-step $H_\infty$ preview controller, which
provides the basis for the convergence results in
Section~\ref{sec:Main}. The formulation follows directly from the
standard full-information $H_\infty$ synthesis
\cite{zhou96,dullerud99,stoorvogel93,green2012linear} applied to an
augmented system that encodes finite disturbance preview.

A $\gamma$-suboptimal $H_\infty$ controller with $p$-step preview can be
constructed by augmenting the state to include future disturbances.
Related preview results in both finite- and infinite-horizon settings
can be found in
\cite{Wang_Zhang_Xie_2014,Hazell_Limebeer_2008,Liu_Lan_2024,KATAYAMA_HIRONO_1987,Katoh_2004,Kojima_2005,Kojima_2015,Kojima_Ishijima_2003,Kojima_Ishijima_2004,Kojima_Ishijima_2006}.
The augmented state at time $t$ is defined as
\begin{align}
\label{eq:AugStates}
\hat{x}_p(t)  :=  
\bmtx
x(t)\\ d(t)\\ \vdots\\ d(t+p-1)
\emtx \in \R^{n_x + p n_d}.
\end{align} The resulting augmented plant is:
\begin{align}
  \label{eq:PicAug}
   \hat{x}_p(t+1) =
    \hat{A}_p\hat{x}_p(t) + \hat{B}_{d,p} d(t+p)+ \hat{B}_{u,p} u(t),
\end{align}
where $d(t+p)$ denotes the $p$-step ahead disturbance available through preview at time $t$. The system matrices are:

{\small
\begin{align}
\label{eq:AugStateMatric}
\hat{A}_p &=
\bmtx
A & B_d & 0 & \cdots & 0\\
0 & 0 & I & \ddots & \vdots \\
\vdots & \ddots & \ddots & \ddots & 0 \\
\vdots & \ddots & \ddots & \ddots & I \\
0 & \cdots & \cdots & \cdots & 0
\emtx,
&
\hat{B}_{u,p} &=
\bmtx
B_u\\ 0\\ \vdots\\ 0
\emtx,
&
\hat{B}_{d,p} &=
\bmtx
0\\ \vdots\\ 0\\ I
\emtx.
\end{align}
}
The cost can be rewritten in terms of the augmented state:
\begin{align}
\label{eq:AugCost}
J(K_p,d) 
= \sum_{t=0}^\infty 
\hat{x}_p(t)^\top \hat{Q}_p \hat{x}_p(t) + u(t)^\top R u(t),
\end{align}
where $\hat{Q}_p  :=  \mathrm{blkdiag}(Q,0,\ldots,0)$ is the state cost matrix.

Using the augmented state $\hat{x}_p(t)$ in \eqref{eq:AugStates}, the
$p$-step preview problem can be formulated as a full-information $H_\infty$
control problem on the augmented dynamics \eqref{eq:PicAug}, where the
exogenous input is $d(t+p)$ and the measured signal is $\hat{x}_p(t)$.
Given $\gamma>0$, a controller is said to be $\gamma$-suboptimal if it
stabilizes the augmented closed loop and satisfies
\begin{align}
\label{eq:hinfgoal_preview}
J(K_p,d) < \gamma^2\| d \|_2^2 
\qquad \forall d \in \ell_2,\ d\ne 0 .
\end{align}

\begin{theorem}
\label{thm:HinfPreview}
Fix $p\ge 0$ and let $(\hat{A}_p,\hat{B}_{u,p},\hat{B}_{d,p},\hat{Q}_p,R)$ be
the augmented data defined in \eqref{eq:PicAug}--\eqref{eq:AugCost} and satisfy conditions $(i)-(iv)$ in Theorem~\ref{thm:NC}. For any $\gamma>0$, a $\gamma$-suboptimal $p$-step preview controller exists if and only if there
exists a symmetric matrix $\hat{X}_p$ such that:
\begin{itemize}
\item[1)] $\hat{X}_p$ satisfies the DARE defined by
$(\hat{A}_p,\hat{B}_p,\hat{Q}_p,\hat{R})$ where
$\hat{B}_p := [\hat{B}_{u,p},\,\hat{B}_{d,p}]$ and
$\hat{R} := \mathrm{diag}(R,-\gamma^2 I)$.
\item[2)] $ H := R+\hat{B}_{u,p}^{\top}\hat{X}_p\hat{B}_{u,p}\succ 0$.
\item[3)] $
\Delta_{p}\coloneq \gamma^2I-\hat{B}_{d,p}^{\top}\hat{X}_{p}\hat{B}_{d,p}
+\hat{B}_{d,p}^{\top}\hat{X}_{p}\hat{B}_{u,p}H^{-1}\hat{B}_{u,p}^{\top}\hat{X}_{p}\hat{B}_{d,p}\succ 0$
\end{itemize}
Moreover, when these conditions hold, one $\gamma$-suboptimal preview policy is
\begin{align}
\label{eq:HinfPreviewPolicy_aug}
u_{\infty,p}(t) = -\hat{K}_{x,p}\,\hat{x}_p(t) - \hat{K}_{d,p}\, d(t+p),
\end{align}
where $\hat{K}_{x,p} = H^{-1}\hat{B}_{u,p}^{\top}\hat{X}_p \hat{A}_p$ and 
$\hat{K}_{d,p} = H^{-1}\hat{B}_{u,p}^{\top}\hat{X}_p \hat{B}_{d,p}$.
\end{theorem}

Theorem~\ref{thm:HinfPreview} follows directly by applying Theorem~9.2 of
\cite{stoorvogel93} to the augmented plant \eqref{eq:PicAug} with generalized
error $e(t)=\hat{C}_p\hat{x}_p(t)+Du(t)$, where
$\hat{C}_p := \bmtx \hat{Q}_p^{1/2} & 0 \emtx^\top$ and
$D := \bmtx 0 & R^{1/2}\emtx^\top$.
With this choice, $J(K_p,d)=\|e\|_2^2$ and the closed-loop induced $\ell_2$
gain from $d$ to $e$ is less than $\gamma$ if and only if
\eqref{eq:hinfgoal_preview} holds.

This policy can be further expressed directly in terms of the plant state $x(t)$ and previewed disturbances:
\begin{align}
\label{eq:HinfpreviewRe}
u_{\infty,p}(t) &= -\tilde{K}_{x}\,x(t)  - \tilde{K}_d\,d(t)-
\tilde{K}_v\sum_{j=1}^{p}\hat{X}_{p}^j\,d(t+j),
\end{align}
% \begin{align}
% \label{eq:HinfpreviewRe}
% u_{\infty,p}(t) &= -\tilde{K}_{x}\,x(t) - \tilde{K}_{d}\,d(t) -
% \tilde{K}_vv_{\infty,p}(t+1)\\
% v_{\infty,p}(t+1)&=\sum_{j=1}^{p}\hat{X}_{p,j,d}\,d(t+j),
% \end{align}
where
\begin{align}
\tilde{K}_{x} &= H^{-1} B_u^{\top}\hat{X}_{p}^0 A, \\
\tilde{K}_{d} &= H^{-1} B_u^{\top}\hat{X}_{p}^0 B_d, \\
\tilde{K}_{v} &= H^{-1} B_u^{\top}.
\end{align}
Here, we block partition the first $n_x$ rows of $\hat{X}_p$ as follows:
\begin{align}
   \bmtx
        \hat{X}_{p}^0,\hat{X}_{p}^1,\dots,\hat{X}_{p}^p
    \emtx,
\end{align}
where $\hat{X}_{p}^0\in \R^{n_x \times n_x}$ and $\hat{X}_{p}^j\in \R^{n_x \times n_d}$ for each $j=1,\dots,p$. The minimal achievable $\gamma$ with $p$ steps of preview is denoted $\gamma_{\infty,p}$.

Thus, $u_{\infty,p}(t)$ uses the preview information $i_p(t)$ in the same
additive structure as the non-causal controller, but truncated to $p$
steps.

\section{Main Results}
\label{sec:Main}
In this section, we first recall the $H_2$ $p$-step preview controller from Theorem~3 of~\cite{liu2025stochastic}.

\begin{theorem}
\label{thm:H2controller}
Let $(A,B_u,B_d,Q,R)$ be given and satisfy conditions $(i)-(iv)$ in
Theorem~\ref{thm:NC}.
% assume: 
% (i) $Q\succeq 0$ and $R \succ 0$, (ii) $\left(A, B_u\right)$ stabilizable, (iii) $A$ is nonsingular, and (iv) $(A, Q)$ has no unobservable modes on the unit circle. 
We also assume the disturbance is the independent and identically distributed (IID) noise: $\E[d_i]=0$ and $\E[d_id_j^{\top}]=\delta_{ij} \, I$ for $i,j\in \{0,1,\ldots\}$, where $\delta_{ij}=1$ if $i = j$ and $\delta_{ij}=0$ otherwise. Define a $H_2$ $p$-step preview controller $K_{2,p}$ with inputs $i_p(t)$ given in \eqref{eq:IHinfo} and output $u_{2,p}(t)$ by the following update equations:
\begin{align}
 \label{eq:H2u}
        u_{2,p}(t)=-K_x x(t)-K_v \sum_{j=t}^{t+p}\left( \tilde{A}^\top \right)^{j-t}X B_d d(j),
\end{align}
where 
\begin{align*}
    K_x & :=  (R+B_u^\top X B_u)^{-1}B_u^\top XA,\\
  K_v &  :=   (R+B_u^\top X B_u)^{-1} B_u^\top,
 % K_d &  :=   (R+B_u^\top X B_u)^{-1} B_u^\top  X B_d.
\end{align*} 
and $X$ is a unique stabilizing positive solution to DARE~\eqref{eq:DARE} defined by $(A,B_u,Q,R)$.

Then this controller $K_{2,p}$ stabilizes the plant and solves:
\begin{align}
    \min_{K_p \text{ stabilizing}} \E\left[ J(K_p,d) \, | \, i_p(0) \right].
\end{align}
\end{theorem}

Note that Theorem~\ref{thm:H2controller} assumes IID noise and optimizes the expected cost in a stochastic formulation. However, we show next that the cost of $H_2$ preview controller with finite preview $p<\infty$ converges, in the limit as $p\to \infty$, to the optimal non-causal controller for any fixed (deterministic) disturbance $d\in \ell_2$.
This result relies on the structural properties of the non-causal controller
established in Theorem~\ref{thm:NC}, and extends our prior work~\cite{liu2025stochastic}
from the stochastic setting to deterministic disturbance sequences.

\begin{lemma}
\label{lemma:H2cost}
Let $(A,B_u,B_d,Q,R)$ be given and satisfy conditions $(i)-(iv)$ in
Theorem~\ref{thm:NC}.
% assume: 
% (i) $Q\succeq 0$ and $R \succ 0$, (ii) $\left(A, B_u\right)$ stabilizable, (iii) $A$ is nonsingular, and (iv) $(A, Q)$ has no unobservable modes on the unit circle. 
Then, for any $d \in \ell_2$,
\begin{align}
\lim_{p\to\infty} J(K_{2,p},d) = J(K_{nc},d).
\end{align}
Moreover, the convergence is uniform in $d$ in the following sense:  For all $\epsilon>0$ there exists $p^*$ such that for all $p\ge p^*$,
\begin{align}
\label{eq:J2JncBound}
    J(K_{2,p},d)-J(K_{nc},d)\leq\epsilon\,\|d\|_2^2.
\end{align}
\end{lemma}
The proof of this lemma is given in the Appendix.
Lemma~\ref{lemma:H2cost} can be used to show that the $H_\infty$ preview controller also converges in cost to the optimal non-causal controller. The formal statement of this result is given in the next theorem.

\begin{theorem}
\label{thm:HinftyNormConv}
Define $\gamma_{nc}$ as the $H_\infty$ norm of the closed-loop using the non-causal controller from Theorem~\ref{thm:NC}.  This corresponds to the minimal $\gamma$ such that $J(K_{nc},d)<\gamma^2 \|d\|_2^2$ for all nonzero $d\in\ell_2$. Similarly, let $\gamma_{\infty,p}$ denote the $H_\infty$ norm of the closed-loop using the $p$-step preview $H_\infty$ optimal controller from Equation~\ref{eq:HinfpreviewRe}. Then
\begin{align}
\lim_{p\to\infty}& \gamma_{\infty,p} \;=\; \gamma_{nc}.
\end{align}
\end{theorem}

\begin{proof}
It follows from Lemma~\ref{lemma:H2cost} that
\begin{align*}
J(K_{2,p},d)
\le \left(\gamma_{nc}^2 + \epsilon\right)\,\|d\|_2^2.
\end{align*}
Taking the supremum over all nonzero $d\in \ell_2$ gives:
\[
\gamma_{2,p}^2 :=  \sup_{d\in\ell_2,\; d\ne 0} \tfrac{J(K_{2,p},d)}{\|d\|_2^2} \le \gamma_{nc}^2 + \epsilon.
\]
The $H_\infty$ controller $K_{\infty,p}$ minimizes the induced $\ell_2$ norm among all $p$-preview controllers,
\[
\gamma_{nc} \le\gamma_{\infty,p} \le \gamma_{2,p} \le \sqrt{\gamma_{nc}^2 + \epsilon}.
\]
% By Lemma~\ref{lemma:H2cost} and the bound in~\eqref{eq:H2CostBound},
% \begin{align*}
% J(K_{2,p},d)
% &\le J(K_{nc},d) + \left(4a + 2b_1+2b_2c_1\right)\,\frac{\alpha^{p+1}}{1-\alpha}\,\|d\|_2^2\\
% &\le \left(\gamma_{nc}^2 + \epsilon_p\right)\,\|d\|_2^2,
% \quad \epsilon_p  :=  \left(4a + 2b_1+2b_2c_1\right)\,\frac{\alpha^{p+1}}{1-\alpha}.
% \end{align*}
% Taking $\sup_{d\neq 0}$ gives
% \[
% \gamma_{2,p}^2 :=  \sup_{d\in\ell_2,\; d\ne 0} \tfrac{J(K_{2,p},d)}{\|d\|_2^2} \le \sqrt{\gamma_{nc}^2 + \epsilon_p}.
% \]
% Since one can always find a $K_{\infty,p}$ that minimizes the induced $\ell_2$ norm among all $p$-preview controllers,
% \[
% \gamma_{nc} \le\gamma_{\infty,p} \le \gamma_{2,p} \le \sqrt{\gamma_{nc}^2 + \epsilon_p}.
% \]
For any $\hat{\epsilon}>0$, one can choose $p$ sufficiently large so that $\sqrt{\gamma_{nc}^2 + \epsilon}<\gamma_{nc}+\hat{\epsilon}$. By the squeeze theorem, this establishes the claim.
\end{proof}

\begin{theorem}
\label{thm:RegretConv}
Let $\gamma_{R,p}$ denote the minimal regret level in Definition~\ref{def:nomregret} achievable by a $p$-step preview controller. Then
\begin{align}
\lim_{p\to\infty}& \gamma_{R,p} \;=\; 0,
\end{align}
and for any $\hat{\epsilon}>0$, there exist $p$ and a controller $K_{R,p}$ such that
\begin{align}
    0\leq J(K_{R,p},d)-J(K_{nc},d)<\hat{\epsilon} \qquad \forall d \in \ell_2, \, d\ne 0.
\end{align}
\end{theorem}

\begin{proof}
From~\eqref{eq:J2JncBound}, for any $\epsilon>0$ there exists $p^*$ such that
\[
\sup_{d\in\ell_2,\; d\ne 0}\frac{J(K_{2,p},d)-J(K_{nc},d)}{\|d\|_2^2}
\leq \epsilon\qquad \forall\, p\geq p^*.
\]

Since one can always find a $K_{R,p}$ that minimizes regret over all $p$-preview controllers,
\[
\gamma_{R,p}^2 \le \sup_{d\in\ell_2,\; d\ne 0}\frac{J(K_{2,p},d)-J(K_{nc},d)}{\|d\|_2^2} \leq \epsilon.
\]
For each fixed $d\in\ell_2$ and $\hat{\epsilon}>0$, one can find a $p$ such that
\[
0 \le J(K_{R,p},d)-J(K_{nc},d) \le \gamma_{R,p}^2\,\|d\|_2^2 \leq \epsilon\|d\|_2^2\leq\hat{\epsilon},
\]
which proves both claims.

% From~\eqref{eq:H2CostBound},
% \[
% \sup_{d\ne 0}\frac{J(K_{2,p},d)-J(K_{nc},d)}{\|d\|_2^2}
% \le
% \left(4a + 2b_1+2b_2c_1\right)\,\frac{\alpha^{p+1}}{1-\alpha} = \epsilon_p.
% \]
% Since one can always find a $K_{R,p}$ that minimizes regret over all $p$-preview controllers,
% \[
% \gamma_{R,p} \;\le\; \epsilon_p \;\xrightarrow[p\to\infty]{}\; 0.
% \]
% For each fixed $d\in\ell_2$ and $\epsilon>0$, one can find a $p$ s.t.
% \[
% 0 \le J(K_p,d)-J(K_{nc},d) \le \gamma_{R,p}\,\|d\|_2^2 \leq \epsilon_p\leq\epsilon,
% \]
% which proves both claims.
\end{proof}

\section{Example}
\label{sec:Example}
This section uses a simple example to demonstrate the performance of preview controllers in both the $H_\infty$ and regret settings. 
\footnote{The code to reproduce all results is available at: 
\url{https://github.com/jliu879/Regret-of-Hinf-Preview-Controllers}.} Consider the LTI system
\begin{align}
\begin{split}
  \label{eq:ExampleSystem}
     x(t+1) &= A x(t) + B_d d(t) + B_u u(t),
\end{split}
\end{align}
with the following parameters:
\begin{align*}
    A &= \bmtx3&1\\-1&-2\emtx, \quad 
    B_d  = \bmtx1\\1\emtx, \quad
    B_u = \bmtx3\\-1\emtx.
\end{align*}
The quadratic cost parameters are: $Q = \bmtx3&0\\0&3\emtx$, $R = 1$.
Three controllers are synthesized for this system:

1. $H_\infty$ Preview Controller:  
   For each preview horizon $p$, we compute the optimal $H_\infty$ $p$-step preview controller $K_{\infty,p}$. 
    This controller is computed using the approach outlined in Section~\ref{sec:HinfPreview}.  We use bisection to minimize $\gamma$ within a small tolerance of $10^{-10}$.  We denote the minimal (nearly optimal) feasible performance as $\gamma_{\infty,p}$.
   
2. Non-Causal Controller: 
   The optimal non-causal controller $K_{nc}$ is designed according to Theorem~\ref{thm:NC}. We let $\gamma_{nc}$ denote the closed-loop $H_\infty$ performance achieved by the optimal non-causal controller.
   % We compare the $H_\infty$-norm of the closed loops induced by $K_{\infty,p}$ and $K_{nc}$ (see Fig.~\ref{fig:Hinfpreview}).

3. Regret Preview Controller:  
   Finally, we synthesize the regret-optimal $p$-step preview controller $K_{R,p}$. This controller is computed using the approach outlined in \cite{liu2024robust} and Section~\ref{sec:HinfPreview}. The non-causal controller is first computed as in Theorem~\ref{thm:NC}. Then a spectral factorization is used to transform the preview regret problem into an equivalent $H_\infty$ preview formulation. Finally, the transformed problem is solved as an $H_\infty$ preview problem using the same synthesis approach as above. Bisection is again used to minimize $\gamma$ to within a tolerance of $10^{-10}$. The minimal (nearly optimal) feasible performance is denoted by $\gamma_{R,p}$.

First, we consider the closed-loop performance obtained by the (nearly) optimal $p$-step preview controller,  $K_{\infty,p}$. Figure~\ref{fig:Hinfpreview} shows $\gamma_{\infty,p}$ versus the preview length $p$ (blue dashed curve).  The performance of the optimal non-causal controller, $K_{nc}$, is also shown (red solid). The performance of the preview controller $\gamma_{\infty,p}$ improves monotonically with increasing $p$ and approaches the non-causal bound $\gamma_{nc}$. This illustrates our result in Theorem~\ref{thm:HinftyNormConv} that $\gamma_{\infty,p} \to \gamma_{nc}$  as $p\to \infty$. Preview information enables the controller to asymptotically match the $H_\infty$ performance of the non-causal baseline.

\begin{figure}[ht]
    \centering
    \includegraphics[width=0.5\linewidth]{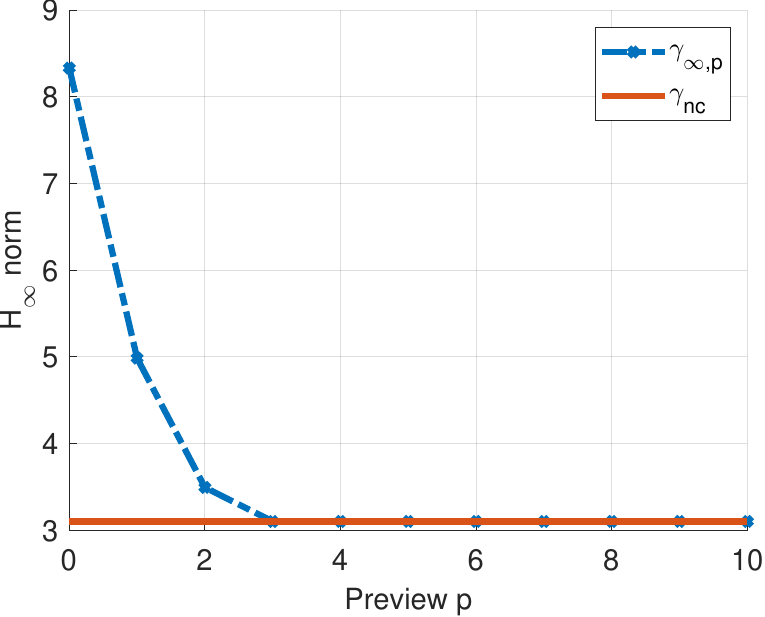}
    \caption{Closed-loop $H_\infty$ norm of the $p$-step preview controller versus preview length $p$. The closed-loop $H_\infty$ norm with the optimal non-causal controller is also shown.}
    \label{fig:Hinfpreview}
\end{figure}

Next, we analyze the regret performance of the additive regret controller $K_{R,p}$ relative to the non-causal baseline $K_{nc}$. Figure~\ref{fig:Regpreview} shows the optimal regret bound $\gamma_{R,p}$ as a function of preview length $p$. The regret decreases rapidly as $p$ increases. This illustrates the result stated in Theorem~\ref{thm:RegretConv}:
$\gamma_{R,p}\to 0$ as $p\to \infty$. Hence, preview information enables the controller to asymptotically match the performance of the non-causal baseline in terms of regret as well.

\begin{figure}[ht]
    \centering
    \includegraphics[width=0.5\linewidth]{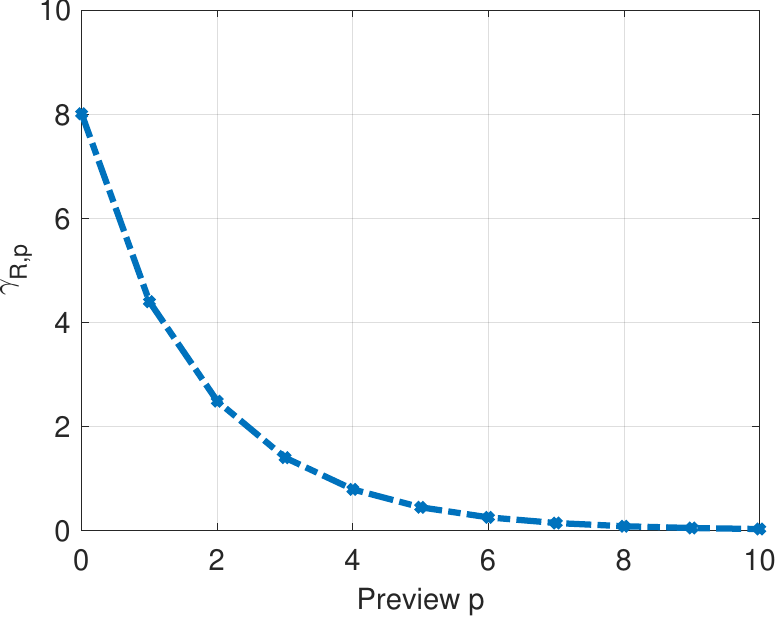}
    \caption{Optimal additive regret $\gamma_{R,p}$ versus preview length $p$.}
    \label{fig:Regpreview}
\end{figure}

Finally, the regret bound is always larger than the excess $H_\infty$ gap, i.e., $\gamma_{R,p} \geq \gamma_{\infty,p} - \gamma_{nc}$. Specifically, the following inequality holds for any stabilizing controller $K_p$:
\begin{align*}
    \sup_{\|d\|<1}[J(K_p,d)-J(K_{nc},d)]\geq \sup_{\|d\|<1}J(K_p,d)-\sup_{\|\hat{d}\|<1}J(K_{nc},\hat{d}).
\end{align*}
Minimizing both sides over stabilizing controllers $K_p$ yields $\gamma_{R,p}^2 \geq \gamma_{\infty,p}^2 - \gamma_{nc}^2$. Since $\gamma_{\infty,p}>\gamma_{nc}>0$, it follows that $\gamma_{\infty,p}^2 - \gamma_{nc}^2\geq(\gamma_{\infty,p} - \gamma_{nc})^2$. As a result, $\gamma_{R,p} \geq \gamma_{\infty,p} - \gamma_{nc}$ can be shown. Therefore, the convergence of regret to zero  (Figure~\ref{fig:Regpreview}) is slower than the convergence of the $H_\infty$ norm (Figure~\ref{fig:Hinfpreview}).
Together, these results highlight that increasing preview length improves both $H_\infty$ performance and regret guarantees, and that both metrics converge to their non-causal limits as $p \to \infty$.

\section{CONCLUSIONS}
This paper studied preview control in the $H_\infty$ and regret-optimal frameworks. We showed that the $p$-step $H_\infty$ preview controller achieves closed-loop performance that converges to the non-causal bound as $p \to \infty$. We also proved that the optimal regret of the preview controller converges to zero. This demonstrates that preview enables controllers to asymptotically match non-causal performance. A numerical example confirmed these results. Future work will establish additional convergence properties, including convergence of the $H_\infty$ preview controller to the non-causal controller.

%%%%%%%%%%%%%%%%%%%%%%%%%%%%%%%%%%%%%%%%%%%%%%%%%%%%%%%%%%%%%%%%%%%%%%%%%%%%%%%%

%%%%%%%%%%%%%%%%%%%%%%%%%%%%%%%%%%%%%%%%%%%%%%%%%%%%%%%%%%%%%%%%%%%%%%%%%%%%%%%%

%%%%%%%%%%%%%%%%%%%%%%%%%%%%%%%%%%%%%%%%%%%%%%%%%%%%%%%%%%%%%%%%%%%%%%%%%%%%%%%%

\section*{ACKNOWLEDGMENT}

%%%%%%%%%%%%%%%%%%%%%%%%%%%%%%%%%%%%%%%%%%%%%%%%%%%%%%%%%%%%%%%%%%%%%%%%%%%%%%%%

This material is based upon work supported by the National Science Foundation under Grant No. 2347026. Any opinions, findings, and conclusions or recommendations expressed in this material are those of the author(s) and do not necessarily reflect the views of the National Science Foundation. 
\bibliographystyle{IEEEtran}
\bibliography{reference}

\appendix
Here we provide the detailed proof of Lemma~\ref{lemma:H2cost}.
\begin{proof}
 By Theorem~3 of~\cite{liu2025stochastic}, the resulting cost difference can be expressed as:
{\footnotesize
\begin{align}
    \label{eq:H2CostDiff}
        J(K_{2,p},d)-J(K_{nc},d)=
        \sum_{t=0}^\infty  &\Big\{\underbrace{-2d(t)^{\top}B_d^\top\sum_{l=p+1}^\infty(\tilde{A}^\top)^lXB_dd(t+l)}_{l(t)}\\\nonumber
        &+\underbrace{2x(t)^\top (\tilde{A}^\top)^{p+1}XB_dd(t+p)}_{m(t)}\\ \nonumber
        &+\underbrace{\begin{aligned}[t]
        [2K_v & \sum_{j=t}^{t+p}\left( \tilde{A}^\top \right)^{j-t}
        X B_d d(j)
        +K_v\sum_{j=t+p+1}^{\infty}\left( \tilde{A}^\top \right)^{j-t}
        X B_d d(j)]^\top 
        H[\sum_{l=p+1}^\infty(\tilde{A}^\top)^lXB_dd(t+l)]
        \end{aligned}}_{q(t)} \Big\}.
        % &+\underbrace{[2K_dd(t)+K_vv_{2,p}(t+1)+K_vv_{nc}(t+1)]^\top H[\sum_{l=p+1}^\infty(\tilde{A}^\top)^lXB_dd(t+l)]}_{q(t)}\}.
    \end{align}
}
Define the finite constants
{\small
\begin{align}
    a& :=  \|B_d^\top\|_{2\to 2}\| X\|_{2\to 2}\| B_d\|_{2\to 2}<\infty,\\    
        b& := \|B_d^\top\|_{2\to 2}\| X\|_{2\to 2}\|K_v^\top\|_{2\to 2}\| H\|_{2\to 2}
       \|X\|_{2\to 2}\|B_d\|_{2\to 2}<\infty.
\end{align} }
Using Cauchy--Schwarz, reindexing of sums, and the fact that $x(t)$ is linear time-invariant convolutions of $\{d(j)\}$, one can derive the following bounds:
{\small
\begin{align*}
\sum_{t=0}^\infty |l(t)|
&\le2a\sum_{l=p+1}^\infty \left(\|(\tilde{A}^\top)^l\|_{2\to 2}\sum_{t=0}^\infty\|d(t)\|\|d(t+l)\|\right)\\
&\le 2a\left(\sum_{l=p+1}^\infty \|(\tilde{A}^\top)^l\|_{2\to 2}\right)\|d\|_2^2,
\end{align*}
\begin{align*}
\sum_{t=0}^\infty |m(t)|
&\le 2a\sum_{t=0}^\infty\sum_{j=0}^{t-1} \left(\|(\tilde{A}^\top)^{t-j+p}\|_{2\to 2}\|d(j)\|\|d(t+p)\|\right)\\
&\le 2a\sum_{j=0}^\infty\sum_{t=j+1}^{\infty}\left( \|(\tilde{A}^\top)^{t-j+p}\|_{2\to 2}\|d(j)\|\|d(t+p)\|\right)\\
&\underbrace{\le}_{l :=  t-j+p}2a\left(\sum_{l=p+1}^\infty \|(\tilde{A}^\top)^l\|_{2\to 2}\right)\|d\|_2^2,
\end{align*}}
{\footnotesize
\begin{align*}
\sum_{t=0}^\infty |q(t)|
 &\leq 2b\sum_{t=0}^\infty\sum_{j=t}^\infty\sum_{l=p+1}^\infty\left(\|(\tilde{A}^\top)^{j-t}\|_{2\to2}\|(\tilde{A}^\top)^l\|_{2\to2}\|d(j)\|\|d(t+l)\|\right)\\
\underbrace{\leq}_{k :=  j-t}& 2b\sum_{k=0}^\infty\sum_{l=p+1}^\infty\left(\|(\tilde{A}^\top)^{k}\|_{2\to2}\|(\tilde{A}^\top)^l\|_{2\to2}\sum_{t=0}^\infty\|d(t+k)\|\|d(t+l)\|\right)\\
\leq  2b &\left(\sum_{k=0}^\infty \|\tilde{A}^k\|_{2\to 2}\right)\!\left(\sum_{l=p+1}^\infty \|(\tilde{A}^\top)^l\|_{2\to 2}\right)\|d\|_2^2.
\end{align*}
}
In the bounds for $l(t)$, $m(t)$, and $q(t)$, we interchange the order of
summation with respect to the time indices $t$, $j$, and $l$.
This is justified as follows.
After taking absolute values and operator norms, all summands become
nonnegative. Hence, by Tonelli's theorem, the order of summation over
the corresponding countable index sets can be exchanged.

Moreover, for any fixed integer $l \ge 0$, the inner sums satisfy
\begin{align*}
    \sum_{t=0}^\infty \|d(t)\|\,\|d(t+l)\|
\le
\Big(\sum_{t=0}^\infty \|d(t)\|^2\Big)^{1/2}
\Big(\sum_{t=0}^\infty \|d(t+l)\|^2\Big)^{1/2}
\le \|d\|_2^2,
\end{align*}

by the Cauchy--Schwarz inequality and the fact that
$\sum_{t=0}^\infty \|d(t+l)\|^2 \le \|d\|_2^2$.
Since $\tilde A$ is Schur, $\sum_{l=p+1}^\infty \|(\tilde A^\top)^l\|_{2\to2} < \infty$,
which guarantees that all rearranged series are finite.

To obtain an explicit convergence rate, we invoke Gelfand’s formula.
Since $\tilde A$ is Schur, there exists $\alpha \in (\rho(\tilde A),1)$
and an integer $T$ such that $\|\tilde A^j\|_{2\to2} \le \alpha^j$ for all
$j \ge T$. Consequently, for $p \ge T$,
\[
\sum_{l=p+1}^\infty \|(\tilde{A}^\top)^l\|_{2\to 2}\;\le\;\frac{\alpha^{p+1}}{1-\alpha}
\quad\text{and}\quad
c :=  \sum_{k=0}^\infty \|\tilde{A}^k\|_{2\to 2} < \infty.
\]
Combining the three bounds yields
\begin{align}
\label{eq:H2CostBound}
0 \le J(K_{2,p},d) - J(K_{nc},d)
\le \left(4a + 2bc\right)\,\frac{\alpha^{p+1}}{1-\alpha}\,\|d\|_2^2,
\end{align}
which vanishes as $p\to\infty$ for each fixed $d\in\ell_2$. In fact, the convergence is uniform if we restrict to $\|d\|_2\leq1$.

\end{proof}

\addtolength{\textheight}{-12cm}   % This command serves to balance the column lengths

\end{document}